\documentclass{article}
\usepackage{amssymb,amsmath}
\usepackage{graphicx}
\usepackage{amsthm}

\def\qed{\hfill {\hbox{${\vcenter{\vbox{               
   \hrule height 0.4pt\hbox{\vrule width 0.4pt height 6pt
   \kern5pt\vrule width 0.4pt}\hrule height 0.4pt}}}$}}}

\newtheorem*{theorem*}{Theorem}
\newtheorem{theorem}{Theorem}
\newtheorem{definition}{Definition}
\newtheorem{lemma}[theorem]{Lemma}
\newtheorem{proposition}[theorem]{Proposition}

\newcommand{\rst}[1]{\ensuremath{{\mathbin\vert}%
\raise-.5ex\hbox{$#1$}}}

\date{}

\title{\Large \textbf{Simple, locally finite dimensional Lie algebras in positive characteristic}}

\author{Johanna Hennig} 

\begin{document}


\maketitle

\begin{abstract}
We prove two structure theorems for simple, locally finite dimensional Lie algebras over an algebraically closed field of characteristic $p$ which give sufficient conditions for the algebras to be of the form $[R^{(-)}, R^{(-)}] / (Z(R) \cap [R^{(-)}, R^{(-)}])$ or $[K(R, *), K(R, *)]$ for a simple, locally finite dimensional associative algebra $R$ with involution $*$. The first proves that a condition we introduce, known as locally nondegenerate, along with the existence of an ad-nilpotent element suffice. The second proves that a uniformly ad-integrable Lie algebra is of this type if the characteristic of the ground field is sufficiently large. Lastly we construct a simple, locally finite dimensional associative algebra $R$ with involution $*$  such that $K(R, *) \ne [K(R, *), K(R, *)]$ to demonstrate the necessity of considering the commutator in the first two theorems.
\end{abstract}


\smallskip


\section{\large \textbf{Introduction}}

An important result which characterizes the simple, finite dimensional Lie algebras over a field of positive characteristic is the Kostrikin-Strade-Benkart theorem (\cite{K2}, \cite{S}, \cite{Be}): Suppose $L$ is a simple, finite dimensional Lie algebra over a field $F$ of characteristic $p > 5$. If $L$ is nondegenerate and there exists a nonzero element $x \in L$ such that $ad(x)^{p-1}$ = 0, then $L$ is a Lie algebra of classical type. This theorem was improved by Premet in \cite{P}: Every finite dimensional, nondegenerate simple Lie algebra over an algebraically closed field of characteristic $p > 5$ is classical. 

\bigskip

\noindent Suppose $L$ is a simple, infinite dimensional Lie algebra over an algebraically closed field of characteristic zero. In \cite{BBZ}, Bahturin, Baranov, and Zalesskii prove that $L$ embeds into a locally finite associative algebra if and only if $L$ is isomorphic to $[K(R, *), K(R, *)]$ where $*$ is an involution and $R$ is an involution simple locally finite associative algebra. This utilizes and extends earlier work of Baranov in \cite{B}. In \cite{Za}, Zalesskii asks for a characterization of such Lie algebras when the ground field is algebraically closed of positive characteristic. 

\bigskip

\noindent The following result extends the Kostrikin-Strade-Benkart theorem to the infinite dimensional case and also addresses Zalesskii's question in the spirit of this theorem.

\bigskip

\begin{theorem}
Let $L$ be a simple, infinite dimensional, locally finite Lie algebra over an algebraically closed field $F$ of characteristic $p > 7$ or characteristic zero. Then the following conditions are equivalent:
\begin{enumerate}
\item $L$ is locally nondegenerate and there is some nonzero element $x \in L$ such that $ad(x)^{p-1} = 0$. 
\item $L \cong [R^{(-)}, R^{(-)}] / (Z(R) \cap [R^{(-)}, R^{(-)}])$ where $R$ is a locally finite, simple associative algebra or $L \cong [K(R, *), K(R, *)]$ where $R$ is a locally finite, simple associative algebra with involution $*$.
\end{enumerate}
\end{theorem}

\noindent For the definition of locally nondegenerate, see section \ref{LND}. We note that over a field of characteristic zero, the condition that $L$ be locally nondegenerate is trivial. Also, when the characteristic of the ground field is zero, the condition that $ad(x)^{p-1} = 0$ for some $x$ simply means that there exists some ad-nilpotent element (the index of nilpotence is inconsequential). Thus, our proof of Theorem 1 also provides a new proof of the main theorem in \cite{BBZ}.

\bigskip

\noindent Another result of this kind does not assume $L$ is locally nondegenerate, but imposes a stronger assumption on the characteristic of the ground field. We say $L$ is \textit{uniformly ad-integrable} if the adjoint representation is uniformly integrable, that is, for each element $a \in L$, $ad(a)$ is a root of some polynomial $f_a(t) \in F[t]$. Note that a simple, locally finite Lie algebra embeds into a locally finite associative algebra if and only if it is uniformly ad-integrable. For $a \in L$, let $\mu_a(t)$ be the minimal polynomial of $ad(a)$ and define $d(L) = min\{ \mbox{deg $\mu_a(t)$} |  a\mbox{ is not ad-nilpotent} \}$. 

\bigskip

\begin{theorem}
Let $L$ be a simple, locally finite Lie algebra which is ad-integrable over an algebraically closed field $F$ of characteristic $p>3$. If $p > 2d(L) - 2$, then $L \cong [R^{(-)}, R^{(-)}] / (Z(R) \cap [R^{(-)}, R^{(-)}])$ where $R$ is a locally finite, simple associative algebra or $L \cong [K(R, *), K(R, *)]$ where $R$ is a locally finite, simple associative algebra with involution $*$.
\end{theorem}

\bigskip

\noindent The condition that $p > 2d(L)-2$ is trivial when $F$ has characteristic zero, and so Theorem 2 also reduces to the result of \cite{BBZ} in this case.

\bigskip

\noindent In \cite{H}, Herstein asks if there is a simple associative algebra $R$ with involution $*$ such that $K(R, *) \ne [K(R, *), K(R, *)]$. An example of a division ring with this property was provided in \cite{L}. However, this example is not locally finite. In the final section, we prove the following theorem.

\bigskip

\begin{theorem}\label{example}
There exists a simple, locally finite associative algebra $R$ with involution $*$ so that $K(R, *) \ne [K(R, *), K(R, *)]$.
\end{theorem}

\bigskip

\noindent This demonstrates the necessity of considering the commutator $[K(R, *), K(R, *)]$ in Theorem 1 and Theorem 2.

\section{\large \textbf{Preliminaries}}

\noindent An algebra $A$ is \textit{locally finite dimensional} (\textit{locally finite}) if any finitely generated subalgebra is finite dimensional. It is straightforward to check that a subalgebra or quotient of a locally finite algebra is also locally finite. An algebra is \textit{locally nilpotent} if every finitely generated subalgebra is a finite dimensional nilpotent algebra. An algebra is \textit{locally solvable} if every finitely generated subalgebra is a finite dimensional solvable algebra. The following lemma is well-known.

\begin{lemma} \label{loc nilp}
A locally nilpotent algebra cannot be simple. A locally solvable Lie algebra cannot be simple.
\end{lemma}

\noindent Let $R$ be an associative algebra. We use the notation $R^{(-)}$ to denote the Lie algebra induced by the commutator of the associative product. It is straightforward to check that if $R$ is locally finite as an associative algebra, then $R^{(-)}$ is locally finite as a Lie algebra. In \cite{H}, Herstein proves the following: 
Suppose $R$ is a simple associative algebra with center $Z$ over a field $F$ of characteristic other than $2$ and $R$ is not 4-dimensional over $Z$. Then $[R^{(-)}, R^{(-)}] / (Z \cap [R^{(-)}, R^{(-)}])$ is a simple Lie algebra.

\bigskip

\noindent Let $A$ be an arbitrary algebra over a field $F$. An involution $*: A \rightarrow A$ is a linear transformation such that $(a^*)^* = a$ and $(ab)^* = b^*a^*$ for any $a$, $b$ in $A$. If $*$ is an involution of an associative algebra $R$, the set of skew-symmetric elements $K(R, *) = \{ x \in R | x^* = -x \}$ is a Lie algebra. In \cite{H}, the following theorem is also proved: 
Suppose $R$ is a simple associative algebra over $F$ of characteristic other than $2$ with involution $*$ and $R$ is more than 16-dimensional over its center $Z$. If $K = K(R, *)$ then $[K,K]/(Z\cap [K,K])$ is a simple Lie algebra.

\bigskip

\noindent The following two lemmas are classical results which appear in \cite{BZ} as lemmas 2.4 and 2.3, respectively. We do not assume the ground field $F$ is of characteristic zero, as in \cite{BZ}. However, $F$ is perfect since $F$ is algebraically closed, so we can apply the Wedderburn Principle Theorem in the proof of lemma \ref{levi} and otherwise the proof remains the same as presented in \cite{BZ}.

\begin{lemma}\label{levi}
Suppose $A$ is a finite dimensional, associative algebra with involution $*: A \rightarrow A$ over an algebraically closed field $F$ of characteristic other than 2. Let $N$ be the radical of $A$, which is the largest nilpotent ideal. Then $A = B \oplus N$, where $B \cong A/N$ is a semisimple subalgebra invariant under $*$. 
\end{lemma}

\begin{lemma}\label{matrix}
Suppose $B$ is a semisimple, finite dimensional associative algebra over an algebraically closed field $F$ with involution $*: B \rightarrow B$. Then $K(B,*) \cong \bigoplus_{i \in \Omega_0} K(M_{n_i}(F), *) \bigoplus_{i \in \Omega_1} M_{n_i}(F)^{(-)}$, where the involution $*$ restricted to $M_{n_i}(F)$ is transposition or the symplectic involution.
\end{lemma}

\subsection{\large \textbf{$\mathbb{Z}$-graded algebras}}\label{GL}

A $\mathbb{Z}-$grading of an algebra $A$ is a decomposition into a sum of subspaces $A = \sum_{i \in \mathbb{Z}} A_i$ such that $A_iA_j \subseteq A_{i+j}$. Such a grading is finite if the set $\{ i \in \mathbb{Z}| A_i \ne (0)\}$ is finite and the grading is nontrivial if $\sum_{i \ne 0} A_i \ne (0)$.

\begin{lemma}\label{simplegraded}
Suppose $R = R_{-n} + \dots + R_n$ is a simple, associative algebra with nontrivial $\mathbb{Z}$-grading such that $\bigcup_{i \ne 0} R_i \ne 0$. Then $R$ is generated by $\bigcup_{i \ne 0} R_i$.
\end{lemma}

\noindent We note that the same result holds for a simple, graded Lie algebra.

\bigskip

\noindent The following lemma will be used in section \ref{proof1} to show that a locally nondegenerate simple Lie algebra $L$ embeds into a locally finite associative algebra. For a set $X$, we utilize the notation Lie$(X)$ to denote the Lie algebra generated by $X$ and Assoc$(X)$ to denote the associative algebra generated by $X$.

\begin{lemma} \label{loc finite}
Suppose $R$ is an associative algebra over $F$ with a finite nontrivial $\mathbb{Z}$-grading, $R = \sum_{i=-M}^{M} R_i$, where $F$ is of characteristic 0 or $p > M$. Suppose $a_1, \dots, a_n \in R$ such that $a_i \in R_{\alpha_i}$ where $\alpha_i \ne 0$ for all $i$. If Lie$(a_1, \dots, a_n)$ is finite dimensional, then Assoc$(a_1, \dots, a_n)$ is finite dimensional as well. 
\end{lemma}

\begin{proof}
We will show that there is an $N \ge 1$ such that any associative word $a_{i_1} \cdots a_{i_k}$ can be written as a sum of words of the form $\rho_1 \cdots \rho_s$, where each $\rho_i \in$ Lie$(a_1, \dots, a_n)$ and $s \le N$. From this it follows that if dim Lie$(a_1, \dots, a_n) = d$, then dim Assoc$(a_1, \dots, a_n) $ $\le d^{N+1} < \infty$.

\bigskip

\noindent Suppose $w = a_{i_1} \cdots a_{i_k}$ and proceed by induction on the length of $w$. Since $a_ia_j = [a_i, a_j] + a_ja_i$, we can move generators with negative degrees to the left and generators with positive degrees to the right at the expense of introducing commutators, which result in new words of shorter length. Thus $w$ can be written as $w = w_- w_0 w_+ + \sum $ (words of shorter length than $w$).  

\bigskip

\noindent Suppose $w_- = a_{k_1} \dots a_{k_l}$ where $a_{k_i} \in R_{\alpha_{k_i}}$ and $-M \le \alpha_{k_i} < 0$. Then $w_- \in R_{\alpha_{k_1} + \dots + \alpha_{k_l}}$ which becomes $0$ once $\alpha_{k_1} + \dots + \alpha_{k_l} < -M$. Hence, the length of $w_- = l \le M$. Similarly, the length of $w_+$ is also at most $M$. It remains to bound the length of $w_0$.

\bigskip

\noindent Since the generators each have non-zero grading, $w_0 = [y_1, y_1'] \dots [y_s, y_s']$ must be a product of commutators, where $|y_i| = -|y_i'|.$ We may write each commutator as a sum of left-normed commutators, so $w_0$ is a sum of elements of the form $[y_1", a_{j_1}] \dots [y_s", a_{j_s}]$ for $y_i" \in $ Lie$(a_1, \dots, a_n)$. Then since $[y_i", a_{j_i}][y_k", a_{j_k}] = [ [y_i", a_{j_i}][y_k", a_{j_k}] ] + [y_k", a_{j_k}][y_i", a_{j_i}]$, we may group commutators with the same righthand generator together, again at the expense of introducing new words of shorter length. Thus without loss of generality we may consider only products of this type: $[x_1, a] \dots [x_s, a]$ for some generator $a$ and $x_i \in $ Lie$(a_1, \dots, a_n)$.

\bigskip

\noindent By repeatedly using the fact that ad$(a)$ is a derivation, one can prove that $[x_1 x_2 \cdots x_s, a, a, \dots, a] = s! [x_1, a] \dots [x_s, a] + \sum ( $words of length $ < s )$. Thus, assuming char $F = 0$ or $p > M$, 

\begin{eqnarray*}
[x_1, a] \dots [x_s, a] = \frac{1}{s!} ( [x_1 x_2 \cdots x_s, a, a, \dots, a] - \sum ( \mbox{words of length} < s  ))
\end{eqnarray*}

\noindent If $|a| = \alpha$, then $|x_i| = -\alpha$, and $|x_1 \dots x_s| = -s\alpha$. Hence, $x_1 \dots x_s = 0$ if $-s\alpha < -M$ or $-s\alpha > M$, i.e. if $s > M$. Therefore,  the length of the product of commutators, $s$, is bounded by $M$. So the length of $w_0$ is less than or equal to $Mn$, where $n$ is the number of generators. Thus the length of $w$ is bounded by $N = M(n+2)$, as desired.
\end{proof}

\subsection{\large \textbf{The Centroid and central polynomials}}\label{C}

\noindent For basic results on the centroid, see \cite{J1}. Let $A$ be an algebra over a field $F$ and let End$_F(A)$ denote the algebra of all linear transformations on $A$.

\bigskip

\noindent Let $\psi_a$ (respectively, $\phi_a$) denote left (right) multiplication by $a$. Let $M(A)$ be the subalgebra of End$_F(A)$ generated by $\{ \psi_a, \phi_a | a \in A\}$. We define the \textit{centroid} $\Gamma$ of $A$ to be the centralizer of $M(A)$ in End$_F(A)$. That is, the elements of $\Gamma$ are linear transformations which commute with all right and left multiplications. If $A$ is a simple algebra, then the centroid $\Gamma$ of $A$ is a field.

\begin{lemma}\label{centroid}
Suppose $L$ is a simple, locally finite Lie algebra over $F$. Suppose $\phi \in \Gamma \cap M(L)$. Then $\phi$ is algebraic over $F$.
\end{lemma}  

\begin{proof}
Suppose $\phi = \sum_{i=1}^{l} \alpha_i ad(a_{i, 1}) \dots ad(a_{i,n_i})$ for $a_{i,j} \in L$ and $\alpha_i \in F$. Let $0 \ne b \in L$. Then $L_1 := $ Lie$(b, a_{i,j})$ is a finite dimensional subalgebra of $L$ and $\left.\phi\right|_{L_1} \in$ End$_F(L_1)$. Hence, there is a polynomial $f(t) \in F[t]$ such that $f(\phi) \cdot L_1 = 0$. Note that $f(\phi) \in \Gamma$. Since $\Gamma$ is a field and $L$ is an algebra over $\Gamma$, we see that $f(\phi) = 0$. 
\end{proof}

\noindent The following proposition will be used at the end of the proof of Theorem 1.

\begin{proposition}
Suppose $L$ is a simple, locally finite Lie algebra over $F$ which is finite dimensional over its centroid, $\Gamma$. Then $\Gamma$ is an algebraic field extension of $F$. 
\end{proposition}

\begin{proof} Let $R = $ End$_{\Gamma}(L) \cong M_d(\Gamma)$, where $d$ is the dimension of $L$ over $\Gamma$. By Formanek and Razmyslov (see \cite{For}, \cite{Raz}), there is a central polynomial $c_d \in M_d(\Gamma)$. That is, $c_d$ is a nonzero polynomial such that $c_d(a_1, \dots, a_n) \in \Gamma$ for all $a_i \in M_d(\Gamma)$. 

\bigskip

\noindent By the Jacobson Density Theorem (see \cite{J1}), $R =$ End$_{\Gamma}(L) = M(L)$, that is, End$_{\Gamma}(L)$ is generated by elements from $\{ad(a) | a \in L\}$. Since $c_d \ne 0$, there must be some $w_1, \dots, w_n \in M(L)$ such that $c := c_d(w_1, \dots, w_n) \ne 0$. Note that $c \in \Gamma$.

\bigskip

\noindent By lemma \ref{centroid}, $c$ is algebraic over $F$. Let $\psi$ be an arbitrary element of $\Gamma$. Then $c\psi \in \Gamma$ and $c\psi \in M(L)$ as well, since if $c = \sum_{i=1}^{l} \alpha_i ad(a_{i, 1}) \dots ad(a_{i,n_i})$, then $c \psi = \sum_{i=1}^{l} \alpha_i ad(\psi a_{i, 1}) \dots ad(a_{i,n_i})$. Therefore, by lemma \ref{centroid}, $c \psi$ is algebraic over $F$. Therefore, since $c$ and $c \psi$ are algebraic, $\psi$ must be algebraic as well. Thus $\Gamma$ is an algebraic field extension of $F$.
\end{proof}

\noindent As a corollary, if $F$ is algebraically closed, then $L$ must be central.

\bigskip



\noindent The next lemma proves that if $F$ is algebraically closed of zero characteristic, then $Z(R) \cap [R^{(-)}, R^{(-)}] = (0)$ and so $[R^{(-)}, R^{(-)}]$ is simple. Suppose $A = R \oplus R^{op}$ with involution $(a, b)^* = (b, a)$. Then $A$ is involution simple and $[K(A, *), K(A, *)] \cong [R^{(-)}, R^{(-)}]$. This explains why one need only consider $[K(A, *), K(A, *)]$ for $A$ an involution simple, locally finite associative algebra in \cite{BBZ} when the ground field has zero characteristic. 

\begin{lemma}\label{char0}
Suppose $R$ is a simple, locally finite Lie algebra over an algebraically closed field $F$ of characteristic zero. Then $Z(R) \cap [R^{(-)}, R^{(-)}] = (0)$.
\end{lemma}

\begin{proof}
Suppose $a \in Z(R) \cap [R^{(-)}, R^{(-)}]$. We wish to show that the ideal in $R$ generated by $a$, denoted $id_{R}(a)$, is locally nilpotent. We can write $a = \sum_i [a_i, b_i]$ for $a_i, b_i \in R$. Let $R_1$ be the finite dimensional subalgebra generated by $\{ a_i, b_i\}$ and suppose $R_2$ is a finite dimensional subalgebra containing $R_1$. Then $a \in Z(R_2) \cap [R_2^{(-)}, R_2^{(-)}]$. Since every element of $R$ is contained in such a subalgebra, it suffices to show that $id_{R_2}(a)$ is nilpotent.

\bigskip

\noindent By lemma \ref{levi}, $R_2 = B_2 \oplus N_2$ where $N_2$ is a nilpotent ideal and $B_2 \cong \overline{R_2}$ is semisimple. Consider $\overline{a} \in \overline{R_2}$. Then $\overline{a} \in Z(\overline{R_2}) \cap [\overline{R_2^{(-)}}, \overline{R_2^{(-)}}]$ and in particular, $\overline{a} \in Z([\overline{R_2^{(-)}}, \overline{R_2^{(-)}}])$. Since $\overline{R_2} \cong B_2$ is semisimple, $\overline{R_2} \cong \bigoplus_{i \in I} M_{n_i}(F)$ and $[\overline{R_2^{(-)}}, \overline{R_2^{(-)}}]\cong \bigoplus_{i \in I} [M_{n_i}(F)^{(-)}, M_{n_i}(F)^{(-)}]$. Since each Lie algebra $[M_{n_i}(F)^{(-)}, M_{n_i}(F)^{(-)}]$ is simple, $Z([\overline{R_2^{(-)}}, \overline{R_2^{(-)}}]) = (0)$. Thus $\overline{a} = 0$ and $a \in N_2$ so $id_{R_2}(a)$ is nilpotent. 

\bigskip

\noindent This implies that $id_{R}(a)$ is locally nilpotent. By lemma \ref{loc nilp}, $a$ cannot be nonzero. Thus $Z(R) \cap [R^{(-)}, R^{(-)}] = (0)$.
\end{proof}


\noindent Lastly, the following lemma will be used in the proof of Theorem 1.

\begin{lemma}\label{pi}
Let $A$ be a simple, locally finite associative algebra and let $\alpha \subseteq A$ be a finite subset. Let $A_{\alpha}$ denote the finite dimensional subalgebra generated by $\alpha$ and let $N_{\alpha}$ be its radical. If $f$ is a polynomial identity such that $f(A_{\alpha}/N_{\alpha}) = (0)$ for every finite subset $\alpha \subseteq A$, then $f(A) = (0)$ as well.
\end{lemma}

\begin{proof}
Let $d$ be the degree of $f$. Suppose to the contrary that $f(a_1, \dots, a_d) \ne 0$ for some set of $a_i \in A$. Let $I$ be the nonzero ideal generated by $f(a_1, \dots, a_d)$. Since $A$ is simple, $I = A$. 

\bigskip

\noindent Choose $x_1, \dots, x_m \in I$. By the definition of $I$, we can write $x_i = \sum_j y_{i,j}f(a_1, \dots, a_d)z_{i,j}$. Let $\alpha$ be the finite set $\{ y_{i,j}, z_{i,j}, a_1, \dots, a_d\}$. By assumption, $f(A_{\alpha}) \subseteq N_{\alpha}$. Since $N_{\alpha}$ is an ideal of $A_{\alpha}$, we have that $x_i \in N_{\alpha}$ for every $i$. Thus, the subalgebra generated by $x_1, \dots, x_m$ is nilpotent, which shows that $I$ is locally nilpotent.

\bigskip

\noindent Thus $I = A$ is locally nilpotent. But this contradicts lemma \ref{loc nilp}. Therefore, we must have that $f(a_1, \dots, a_d) = 0$ for all $a_i \in A$.
\end{proof}

\bigskip

\subsection{\large \textbf{Jordan elements in a Lie algebra}}\label{jordanelements}

\noindent We adopt the notation used in \cite{FGG}. A nonzero element $x$ of a Lie algebra $L$ is a \textit{Jordan element} if $ad(x)^3 = 0$. First, by a theorem of Kostrikin (\cite{K}), a non-zero element which is ad-nilpotent gives rise to a non-zero Jordan element.  

\begin{lemma}[Kostrikin's Descent Lemma]\label{kostrikin}
Suppose $L$ is a Lie algebra over $F$ of characteristic $p \ge 5$ and let $a$ be a nonzero element of $L$ such that $ad(a)^n = 0$ for $4 \le n \le p-1$. 

\smallskip

\noindent Then for every $b \in L$, 
$$
(ad [b, \underbrace{a, \dots, a}_
{\mbox{n-1}}])^{n-1} = 0
$$
\end{lemma}

\bigskip

\noindent Hence, if there is some $0 \ne y \in L$ such that ad$(y)^{p-1} = 0$, then $L$ contains a nonzero Jordan element, $x$. We first state some useful identities concerning Jordan elements which appear in \cite{K} (see also \cite{FGG}). We will utilize the notation $X := ad_x$, where $ad_x(y) = [x, y]$.

\begin{lemma}\label{identities}
Let $x$ be a Jordan element in a Lie algebra $L$ and let $a$ and $b$ be arbitrary elements of $L$. 
\begin{enumerate}
\item[i.] $X^2AX = XAX^2$
\item[ii.] $X^2AX^2 = 0$
\item[iii.] $X^2A^2XAX^2 = X^2AXA^2X^2$
\item[iv.] $[X^2(a), X(b)] = - [X(a), X^2(b)]$
\item[v.] $ad_x^2[[a,x],b] = ad_x^2[[b,x],a]$
\item[vi.] $X^2ad_{[a,X^2(b)]} = ad_{[X^2(a), b]}X^2$
\item[vii.] $ad^2_{X^2(a)}=X^2A^2X^2$
\item[viii.] $ad_x^2(a)$ is a Jordan element.
\end{enumerate}
\end{lemma}

\bigskip

\noindent We now use a Jordan element to create a Jordan algebra associated to $L$ as is done in \cite{FGG}. Recall that a \textit{Jordan algebra}, $J$, is an algebra over a field $F$ with a commutative binary product $\circ$ satisfying the Jordan identity: 
$x^2 \circ (y \circ x) = (x^2 \circ y) \circ x$ for each $x, y \in J$. This product is in general nonassociative. The \textit{triple product} in a Jordan algebra is $\{ x, y, z\} = (x \circ y) \circ z + x \circ (y \circ z) - y \circ (x \circ z)$ and the \textit{U-operator} is the quadratic map $J \rightarrow End_F(J)$ given by $x \mapsto U_x$ where $U_x(y) = \{ x, y, x\}$.

\begin{lemma}\label{jordan}
Let $x$ be a Jordan element of a Lie algebra $L$. Define a new multiplication on $L$ via $a \bullet b : = [[a, x], b]$ and denote this nonassociative algebra by $L^{(x)}$.

\begin{itemize}
\item $ker_L(x) = \{a \in L | [x, [x, a]] = 0 \}$ is an ideal of $L^{(x)}$.
\item $L_x :=  L^{(x)} / ker_L(x)$ is a Jordan algebra with $U$-operator:
\begin{eqnarray*}
U_{\overline{a}}\overline{b} = \overline{ad_a^2ad_x^2b}
\end{eqnarray*}
\end{itemize}
\end{lemma}



\bigskip

\begin{lemma}
Suppose $L$ is a locally finite Lie algebra and $x \in L$ is a Jordan element. Then $L_x$ is a locally finite Jordan algebra.
\end{lemma}

\begin{proof}
It suffices to show that the algebra $L^{(x)}$ with multiplication $a \bullet b : =  [[a, x], b]$ is locally finite. Suppose $a_1, \dots, a_n$ is a finite subset of $L$. By the definition of $\bullet$, the algebra $<a_1, \dots, a_n>$ is contained in the Lie algebra $Lie(a_1, \dots, a_n, x)$ which is finite dimensional.
\end{proof}

\bigskip

\noindent The remaining results in this section can be found in \cite{DFGG}.

\begin{definition}
A nonzero element $e$ in a Lie algebra is called \textup{von Neumann regular} if:
\begin{itemize}
\item $ad_e^3 = 0$, and
\item $e \in ad_e^2(L)$
\end{itemize} 
\end{definition}

\noindent We say that a pair of elements $(e,f)$ is an \textit{idempotent} in a Lie algebra if $ad_e^3 = ad_f^3 = 0$ and $(e, [e,f], f)$ is an $\mathfrak{sl}_2$ triple.

\bigskip

\noindent The following lemma appears as Proposition 2.9 in \cite{DFGG}.

\begin{lemma}
Suppose $e \in L$ is regular.
\begin{itemize}
\item[i.] For every $h \in [e,L]$ such that $[h,e] = 2e$, there exists $f \in L$ such that $[e,f] = h$ and $(e,f)$ is an idempotent.
\item[ii.] Let $(e,f)$ be an idempotent and $h := [e,f]$. Then $ad_h$ is semisimple and its action induces a finite $\mathbb{Z}-$grading on $L$:
\begin{eqnarray*}
L = L_{-2} + L_{-1} + L_0 + L_1 + L_2
\end{eqnarray*}
\end{itemize}
\end{lemma}

\bigskip

\subsection{\large \textbf{Locally nondegenerate Lie algebras}}\label{LND}

\noindent A Lie algebra is \textit{nondegenerate} if there are no nonzero elements $x \in L$ such that $ad_x^2 L = 0$. It is well known that if $R$ is a simple ring, then the Lie algebras $[R^{(-)}, R^{(-)}] / (Z(R) \cap [R^{(-)}, R^{(-)}])$ and $[K(R, *), K(R, *)] / (Z(R) \cap [K(R, *), K(R, *)])$ are nondegenerate (see for instance lemmas 5.2 and 5.13 of \cite{DFGG}).







\bigskip

\noindent A Jordan algebra $J$ is \textit{nondegenerate} if there are no nonzero elements $x \in J$ such that $U_x(J) = \{x, J, x\} = (0)$. Nondegeneracy in Jordan algebras is closely related to the notion of an m-sequence. An \textit{m-sequence} in a Jordan algebra $J$ is a sequence $\{ a_n\}$ such that $a_{n+1} = U_{a_n}(b)$ for some $b \in J$. We say an m-sequence has \textit{length} $k$ if $a_k \ne 0$ and $a_{k+1} = 0$, and we say an m-sequence \textit{terminates} if it has finite length. M-sequences characterize the elements of the \textit{McCrimmon radical}, $M(J)$, which is defined to be the smallest ideal of $J$ inducing a nondegenerate quotient. An element $x \in J$ is in the McCrimmon radical $M(J)$ if and only if any m-sequence which begins with $x$ terminates. Thus, a Jordan algebra $J$ is nondegenerate if and only if there are no nonzero elements $x$ such that every m-sequence beginning with $x$ terminates.

\bigskip

\noindent There is also a notion of m-sequence in Lie algebras which is defined in \cite{GG}: an m-sequence in a Lie algebra $L$ is a sequence $\{ a_n\}$ such that $a_{n+1} = [a_n,[a_n, b_n]]$ for some $b_n \in L$. Recall that for a Lie algebra $L$, the \textit{Kostrikin radical} $K(L)$ is the smallest ideal of $L$ inducing a nondegenerate quotient (\cite{Ze}). If any m-sequence beginning with an element $x$ terminates, then $x \in K(L)$. (The other direction is currently unknown and is the subject of \cite{GG}). 

\bigskip 

\noindent We now define an $S$-sequence in a Lie algebra. For any finite set $S$ in a Lie algebra $L$, an \textit{S-sequence} is a sequence $\{x_n \}$ in $L$ such that for each $n$, $x_{n+1} = [x_n, [x_n, s]]$ for some $s \in S$. Similarly, for a finite set $S$ of a Jordan algebra $J$, an $S$-sequence is a sequence $\{ x_n \}$ in $J$ such that for each $n$, $x_{n+1} = U_{x_n}(s)$ for some $s \in S$. We say an S-sequence has \textit{length} $k$ if $x_k \ne 0$ and $x_{k+1} = 0$, and we say an S-sequence \textit{terminates} if it has finite length. 

\bigskip

\noindent A Lie algebra $L$ is \textit{locally nondegenerate} if there are no nonzero elements $x$ such that for any finite set $S$, every $S$-sequence beginning with $x$ terminates. We define the \textit{local Kostrikin radical}, $K_{loc}(L)$, to be the smallest ideal of $L$ which induces a locally nondegenerate quotient. The following lemma is straightforward.

\begin{lemma}\label{locnd}
If $L$ is a locally nondegenerate Lie algebra, then $L$ is nondegenerate. 
\end{lemma}


\noindent The following proposition shows that the condition of being locally nondegenerate is trivial when the ground field has characteristic zero.

\begin{proposition}
Let $L$ be a simple, locally finite Lie algebra over a field of characteristic zero. Then $L$ is locally nondegenerate.
\end{proposition}

\begin{proof}
Suppose there is some element nonzero $a \in L$ such that for any finite set $S$, any $S$ sequence beginning with $a$ terminates. We will show that the ideal generated by $a$, denoted $id_L(a)$, is nonzero and locally solvable, which will contradict lemma \ref{loc nilp} and complete the proof.

\bigskip

\noindent Suppose $L_1$ is a finite dimensional subalgebra of $L$ containing $a$. It suffices to show that $id_{L_1}(a)$ is solvable. Let $\overline{L_1} = L_1 / Rad(L_1)$. Then $\overline{a}$ is an element of $\overline{L_1}$ such that every $\overline{S}$-sequence beginning with $\overline{a}$ terminates. Therefore, without loss of generality, we may assume $L_1$ is semisimple.

\bigskip

\noindent Thus $L_1$ is a finite dimensional, semisimple Lie algebra over a field of characteristic zero. Thus $L_1$ is nondegenerate, i.e. there are no nonzero elements $x \in L_1$ such that $ad_x^2(L_1) = 0$. Let $S$ be a basis for $L_1$. Since $ad_a^2(L_1) \ne (0)$, there is some element $s_1 \in S$ such that $x_1 = ad_a^2(s_1) \ne 0$. Similarly, $ad^2_{x_1}(L_1) \ne (0)$, so there is some element $s_2 \in S$ such that $x_2 = ad^2_{x_1}(s_2) \ne 0$. Continuing in this way, we can create an $S$-sequence $\{ x_n\}$ which does not terminate for a contradiction. This completes the proof.
\end{proof}

\bigskip

\noindent The proof of the previous lemma also shows that if $L$ is finite dimensional, then $L$ is nondenegerate if and only if $L$ is locally nondegenerate. The next proposition shows that if $R$ is a simple, locally finite associative algebra, then its associated Lie algebras $[R^{(-)}, R^{(-)}] / (Z(R) \cap [R^{(-)}, R^{(-)}])$ and $[K(R, *), K(R, *)]/ (Z(R) \cap [K(R, *), K(R, *)])$ are locally nondegenerate. We first prove a lemma.

\bigskip

\begin{lemma}\label{KB}
Suppose $B$ is a semisimple, finite dimensional associative algebra with involution $*$. Then $K_{loc}([K(B, *), K(B, *)]) \subseteq Z([K(B, *), K(B, *)])$.
\end{lemma}

\begin{proof}
Suppose $a$ is a nonzero element of $[K(B, *), K(B, *)]$ such that every $S$-sequence beginning with $a$ terminates. By lemma \ref{matrix}, 

\begin{eqnarray*}
[K(B,*), K(B,*)] \cong\bigoplus_{i \in \Omega_0} K(M_{n_i}(F), *) \bigoplus_{i \in \Omega_1} [M_{n_i}(F)^{(-)}, M_{n_i}(F)^{(-)}] 
\end{eqnarray*}

\noindent Each summand $K(M_{n_i}(F), *)$ is a simple classical Lie algebra, hence nondegenerate and therefore locally nondegenerate. Therefore, we must have that $a \in \bigoplus_{i \in \Omega_1} [M_{n_i}(F), M_{n_i}(F)] $. Since $[M_{n_i}(F), M_{n_i}(F)]  / Z([M_{n_i}(F), M_{n_i}(F)] )$ is nondegenerate, this implies that $a \in Z([K(B,*), K(B,*)] )$.
\end{proof}

\begin{proposition}\label{locndeg}
Suppose $R$ is a simple, locally finite associative algebra with involution $*$ over an algebraically closed field $F$. Then the Lie algebras $[K(R, *), K(R, *)]/ (Z(R) \cap [K(R, *), K(R, *)])$ and $[R^{(-)}, R^{(-)}] / (Z(R) \cap [R^{(-)}, R^{(-)}])$ are locally nondegenerate.
\end{proposition}

\begin{proof}
It suffices to prove that $L \cong [K(R, *), K(R, *)]/ (Z(R) \cap [K(R, *), K(R, *)])$ is locally nondegenerate when $R$ is an involution simple, locally finite associative algebra. By Theorem 10 from \cite{H}, $L$ is simple. Suppose $\overline{a} \in L$ is an element such that for any finite subset $S \subset L$, every $S$-sequence beginning with $\overline{a}$ terminates. We wish to show that the ideal generated by $\overline{a}$ is locally solvable, which will contradict lemma \ref{loc nilp} unless $\overline{a} = 0$.

\bigskip

\noindent Let $a$ be a preimage of $\overline{a}$ in $[K(R, *), K(R, *)]$ and choose $R_1$ to be a finite dimensional subalgebra of $R$ such that $a \in [K(R_1, *), K(R_1, *)] = L_1$. By lemma \ref{levi}, $R_1 = B_1 \oplus N_1$, so $[K(R_1, *), K(R_1, *)] \subseteq [K(B_1, *), K(B_1, *)] \oplus K(N_1, *)$. Suppose $a'$ is the image of $a$ in $R_1/N_1 \cong B_1$. Then $a'$ is an element of $[K(B_1, *), K(B_1, *)]$  such that every $S$-sequence beginning with $a'$ terminates for every finite subset $S \subset [K(B_1, *), K(B_1, *)]$. Thus by lemma \ref{KB}, $a' \in Z([K(B_1, *), K(B_1, *)])$, which shows that the ideal of $L_1$ generated by $a$ is solvable. This implies that the ideal generated by $\overline{a}$ is locally solvable in $L$, thus $K_{loc}(L) = (0)$.

\end{proof}

\noindent The following proposition will be used in the proof of Theorem 1.

\begin{proposition}\label{Kloc}
Let $x$ be a Jordan element in a Lie algebra $L$. If the Jordan algebra $L_x$ is locally nilpotent, then $x \in K_{loc}(L)$.
\end{proposition}

\begin{proof}
Fix $S= \{s_1, \dots, s_p \}$, a finite set from $L$. Then $\overline{S}$ denotes the finite set $\{\overline{s_1}, \dots, \overline{s_p} \}$ of $L_x$. Suppose $\{ x_n \}$ is an $S$-sequence which begins with $x$. Thus $x_0 = x$ and for each $n \ge 1$, $x_n = [x_{n-1}, [x_{n-1}, s_n]]$ for some $s_n \in S$. 

\bigskip

\noindent We prove by induction that for every $n \ge 1$, there is some $c_n \in$ Lie$(S \cup \{ x \})$ such that $x_n = [x,[x, c_n]]$. For $n = 1$, set $c_1 = s_1$. Suppose the result holds for some $n$, that is, $x_n = ad_x^2(c_n)$. Then $x_{n+1} = [x_n, [x_n, s_n]] = ad^2_{X^2(c_n)}s_n = ad_x^2ad^2_{c_n}ad_x^2(s_n)$ by lemma \ref{identities}. Setting $c_{n+1} = ad^2_{c_n}ad_x^2(s_n)$ finishes the argument.

\bigskip

\noindent Now consider the sequence $\{ \overline{c_n}\}$ in $L_x$. We have that $U_{\overline{c_n}}{\overline{s_n}} = \overline{c_{n+1}}$, thus every element of the sequence is contained in the subalgebra of $L_x$ generated by $\overline{S}$, which is finite dimensional and nilpotent since we assumed $L_x$ is locally nilpotent. Thus, the length of $\{ \overline{c_n}\}$ must be finite and the sequence must terminate. Thus $\overline{c_N} = \overline{0}$ for some $N$ and $x_n = [x,[x,c_n]] = 0$. Therefore, any $S$ sequence beginning with $x$ terminates, and we conclude that $x \in K_{loc}(L)$.
\end{proof}

\section{\large \textbf{Proof of Theorem 1}}\label{proof1}

\noindent This section is devoted to the proof of the following theorem.

\begin{theorem*}
Let $L$ be a simple, infinite dimensional, locally finite Lie algebra over an algebraically closed field $F$ of characteristic $p > 7$ or characteristic zero. Then the following conditions are equivalent:
\begin{enumerate}
\item $L$ is locally nondegenerate and there is some nonzero element $x \in L$ such that $ad(x)^{p-1} = 0$. 
\item $L \cong [R^{(-)}, R^{(-)}] / (Z(R) \cap [R^{(-)}, R^{(-)}])$ where $R$ is a locally finite, simple associative algebra or $L \cong [K(R, *), K(R, *)]$ where $R$ is a locally finite, simple associative algebra with involution $*$.
\end{enumerate}
\end{theorem*} 

\noindent First we prove the backward direction, (2) implies (1). 
\bigskip

\noindent Suppose $R$ is a simple, locally finite associative algebra and $L \cong [R^{(-)}, R^{(-)}] / (Z(R) \cap [R^{(-)}, R^{(-)}])$. By theorem 4 of \cite{H}, $L$ is locally finite and simple. By proposition \ref{locndeg}, $L$ is locally nondegenerate. It remains to show that $L$ contains an ad-nilpotent element.

\bigskip

\noindent Let $\alpha$ be a finite subset of $R$ and let $A_{\alpha}$ denote the finite dimensional, associative algebra generated by $\alpha$. By lemma \ref{levi}, $A_{\alpha} = B_{\alpha} \oplus N_{\alpha}$ where $B_{\alpha}$ is semisimple and $N_{\alpha}$ is nilpotent. Since $F$ is algebraically closed, $B_{\alpha}$ is a direct sum of matrix rings $B_{\alpha} = \oplus M_{n_i}(F)$. 

\bigskip

\noindent If $n_i < 2$ for every $i$ and $\alpha$, then $B_{\alpha}$ is commutative for every $\alpha$ which implies $R$ is commutative by lemma \ref{pi} and L = (0). Hence, we must have that $n_i \ge 2$ for some $\alpha \subset R$. Then $[R^{(-)}, R^{(-)}] \supset [M_n(F)^{(-)}, M_n(F)^{(-)}] \cong \mathfrak{sl}_n(F)$ for $n \ge 2$.

\bigskip

\noindent If $n = 2k$ is even, we set $a := \left( \begin{array}{cc} 0 & I \\ 0 & 0 \end{array} \right)$ where $I$ denotes the $k \times k$ identity matrix. If $n = 2k+1$ is odd, we set $a := \left( \begin{array}{ccc} 0 & I & 0 \\ 0 & 0 & 0 \\ 0 & 0 & 0 \end{array} \right)$, where $I$ denotes the $k \times k$ identity matrix. In either case, $a$ is a nonzero element of $R$ such that $a^2 = 0$, implying that $ad(a)^3 = 0$ in $[R^{(-)}, R^{(-)}]$. Clearly $a$ is not in the center of $R$, thus we have that $ad(\overline{a})^3 = 0$ in $L$.

\bigskip

\noindent Now suppose that $L \cong [K(R, *), K(R, *)]$. Since $F$ is algebraically closed, $Z(R) \cap [K(R, *), K(R, *)] = (0)$. Hence, $L$ is locally finite and simple by Theorem 10 from \cite{H}. Furthermore, $L$ must be locally nondegenerate by proposition \ref{locndeg}. 

\bigskip

\noindent The last step is produce an ad-nilpotent element. As before, let $A_{\alpha}$ denote the finite dimensional, associative algebra generated by $\alpha$, a finite subset of $R$. Without loss of generality we may further assume that $A_{\alpha}$ is invariant under the involution $*$. By lemma \ref{levi}, $A_{\alpha} = B_{\alpha} \oplus N_{\alpha}$ where $B_{\alpha}$ is semisimple and $*-$invariant. As before, we conclude that there is some finite subset $\alpha$ such that $B_{\alpha}$ is isomorphic to a direct sum of matrix rings $B_\alpha = \oplus_i M_{n_i}(F)$ where $n_i \ge 2$ for some $i$. 

\bigskip

\noindent By lemma \ref{matrix}, $K(B_{\alpha},*) \cong \bigoplus_{i \in \Omega_0} K(M_{n_i}(F), *) \bigoplus_{i \in \Omega_1} M_{n_i}(F)^{(-)}$. Therefore, $[K(B_{\alpha},*), K(B_{\alpha},*)]$ must contain a subalgebra isomorphic to $\mathfrak{o}_{n}(F)$, $\mathfrak{sp}_{n}(F)$, or $\mathfrak{sl}_{n}(F)$ for some $n \ge 2$. In each case we can find a nonzero element $a$ such that $a^2 = 0$, implying that $ad(a)^3 = 0$ in $L$. 

\bigskip

\noindent This was already done for $\mathfrak{sl}_{n}(F)$. For $\mathfrak{sp}_{n}(F)$ for $n=2k$, we use the realization of $\mathfrak{sp}_{n}(F)$ as those matrices which preserve the skew-symmetric form given by $\left( \begin{array}{cc} 0 & I_k \\ -I_k & 0 \end{array} \right)$ where $I_k$ denotes the $k \times k$ identity matrix. Then $a := \left( \begin{array}{cc} 0 & B \\ 0 & 0 \end{array} \right)$, where $B$ is any nonzero $k \times k$ symmetric matrix, will be an element of $\mathfrak{sp}_{n}(F)$ whose square is zero. 

\bigskip

\noindent Similarly, we realize $\mathfrak{o}_{n}(F)$ as those matrices which preserve the symmetric form given by $\left( \begin{array}{cc} 0 & I_k \\ I_k & 0 \end{array} \right)$ if $n = 2k$ is even or $\left( \begin{array}{ccc} 1 & 0 & 0 \\ 0 & 0 & I_k \\ 0 & I_k & 0 \end{array} \right)$ if $n=2k+1$ is odd. We set $a := \left( \begin{array}{cc} 0 & C \\ 0 & 0 \end{array} \right)$ if $n = 2k$ is even or $a := \left( \begin{array}{ccc} 0 & 0 & 0 \\ 0 & 0 & C \\ 0 & 0 & 0 \end{array} \right)$ if $n=2k+1$ is odd, where $C$ is any nonzero $k \times k$ skew-symmetric matrix. 

\bigskip

\bigskip

\noindent Now, we prove the forwards direction, (1) implies (2).

\bigskip

\noindent Suppose $L$ is a simple, locally finite Lie algebra over $F$ which is locally nondegenerate and contains some nonzero element $x \in L$ such that ad$(x)^{p-1} = 0$. By the results of section \ref{jordanelements}, we may assume $x$ is a Jordan element, that is, ad$(x)^{3} = 0$, and the Jordan algebra $J = L_x$ is locally finite. 

\bigskip

\noindent Our goal is to prove that $J$ must contain a nonzero idempotent. If we suppose to the contrary that $J$ does not contain a nonzero idempotent, then by lemma 1 of 3.7 in \cite{J2}, every $a \in J$ must be nilpotent, so $J$ is nil. The following is a corollary of a theorem of Albert and can be found in \cite{J2}:
A locally finite, nil Jordan algebra is locally nilpotent.

\bigskip

\noindent Therefore, $J = L_x$ is a locally nilpotent Jordan algebra. By lemma \ref{Kloc}, $x$ must be contained in the local Kostrikin radical, $K_{loc}(L)$. But this is a contradiction since we have assumed that $L$ is locally nondegenerate. Therefore, $J = L_x$ must contain a nonzero idempotent, $\overline{e}^2 = \overline{e}$.

\bigskip

\noindent Let $R_{\overline{e}}$ denote right multiplication by $\overline{e}$ in $J$. Then we have the Pierce decomposition of $J$ into eigenspaces with respect to $R_{\overline{e}}$:

\begin{eqnarray*}
J = \{ \overline{e}, J, \overline{e}\} + \{1- \overline{e}, J, 1-\overline{e}\} + \{ \overline{e}, J, 1- \overline{e}\}
\end{eqnarray*}

\bigskip

\noindent Since $R_{\overline{e}}(\overline{e}) = \overline{e}^2 = \overline{e}$, we see that $\overline{e} \in \{ \overline{e}, J, \overline{e}\}$. Therefore, there is some $\overline{a} \in J$ 
such that $\overline{e} = \{ \overline{e}, \overline{a}, \overline{e}\}$. By lemma \ref{jordan}, $\{ \overline{e}, \overline{a}, \overline{e}\} = ad_{\overline{e}}^2ad_{\overline{x}}^2 \overline{a} = \overline{e}$. Hence, when we lift to the Lie algebra $L$, we have that

\begin{eqnarray*}
X^2E^2X^2 a = ad_x^2ad_e^2ad_x^2 a = ad_x^2e = X^2(e)
\end{eqnarray*}

\bigskip

\noindent By lemma \ref{identities}, $X^2E^2X^2 = ad^2_{X^2(e)}$. Hence $X^2(e) \in ad^2_{X^2(e)}(L)$. 

\bigskip

\noindent Set $e' = ad_x^2e = X^2(e)$. Then we have shown that $e' \in ad^2_{e'}(L)$. Moreover, by lemma \ref{identities}, $e'$ is a Jordan element of $L$. Hence, $e'$ is a von Neumann regular element. 

\bigskip

\noindent By the results of section \ref{jordanelements}, we obtain a finite $\mathbb{Z}$-grading on $L$ via the action of $ad(h)$ for some $h \in [e', L]$:

\begin{eqnarray*}
L = L_{-2} + L_{-1} + L_0 + L_1 + L_2
\end{eqnarray*}

\bigskip

\noindent Therefore, $L$ is a simple, $\mathbb{Z}$-graded Lie algebra over a field of characteristic $p > 7$ (or of characteristic zero). Hence, we can utilize theorem 1 from \cite{Z}:
Since $L = \sum_{i = -2}^{2} L_i$ is a simple graded Lie algebra over a field of characteristic at least $7$ (or of characteristic 0) and $\sum_{i \ne 0} L_i \ne 0$, we have that $L$ is isomorphic to one of the following algebras:
\begin{itemize}
\item[I.] $[R^{(-)}, R^{(-)}] / Z$, where $R =  \sum_{i = -2}^{2} R_i$ is a simple associative $\mathbb{Z}-$graded algebra.
\item[II.] $[K(R, *), K(R, *)] / Z$, where $R =  \sum_{i = -2}^{2} R_i$ is a simple associative $\mathbb{Z}-$graded algebra with involution $*: R \rightarrow R$.
\item[{III.}] The Tits-Kantor-Koecher construction of the Jordan algebra of a symmetric bilinear form.
\item[{IV.}] An algebra of one of the types $G_2, F_4, E_6, E_7, E_8$ or $D_4$.
\end{itemize}

\bigskip

\noindent In cases I and II, it remains to show that the simple, graded associative algebra $R$ is locally finite. Suppose we are in case I, that is, $L \cong [R^{(-)}, R^{(-)}] / (Z(R) \cap [R^{(-)}, R^{(-)}])$. Since $Z(R) \cap [R^{(-)}, R^{(-)}]$ is an ideal which is locally finite and central, $[R^{(-)}, R^{(-)}]$ is a locally finite Lie algebra. Choose a finite set $a_1, \dots, a_n$ from $R$. Since $R$ is generated by $\bigcup_{i \ne 0} R_i$ by lemma \ref{simplegraded}, we may assume that $a_1, \dots, a_n$ are homogeneous of nonzero degree. Note that for $i \ne 0$, $R_i \subseteq [R^{(-)}, R^{(-)}]$ since for any $a \in R_i$, $[h, a] = ia$. Then since $[R^{(-)}, R^{(-)}]$ is locally finite, Lie$(a_1, \dots, a_n)$ is finite dimensional, and so Assoc$(a_1, \dots, a_n)$ is finite dimensional by lemma \ref{loc finite}.  

\bigskip

\noindent Similarly, suppose we are in case II, that is, $L \cong [K(R, *), K(R, *)]$. Again, we may choose a finite set $a_1, \dots, a_n$ from $R$. By theorem 2.13 of \cite{H3}, $L$ generates $R$ as an associative algebra. Therefore, Assoc$(a_1, \dots, a_n) \subseteq$ Assoc$(\{x_i\})$ for some finite set $\{x_i\} \subset L$. By lemma \ref{simplegraded}, we can assume that the $x_i$ each have nonzero degree. Hence, since Lie$(\{x_i\})$ is finite dimensional, we can apply lemma \ref{loc finite} to conclude that Assoc$(\{x_i\})$, and in turn Assoc$(a_1, \dots, a_n)$, is finite dimensional.

\bigskip

\noindent In case III, it has been noted in the literature that $L = TKK(J)$ is isomorphic to $K(R, *)$ for some simple associative algebra $R$ with involution $*$ (\cite{DFGG}, \cite{Z}). We provide an explanation of this fact in order to demonstrate that the associative algebra $R$ is also locally finite.

\bigskip

\noindent Suppose $V'$ is a finite dimensional vector space over a field $F$ equipped with a nondegenerate, symmetric bilinear form $f$ and $J' = F\cdot1 + V' $ is the resulting Jordan algebra. Suppose $W = Fv + Fu + Fw$ is the 3-dimensional vector space equipped with a bilinear form $g$ given by the matrix

\begin{eqnarray*}
\left( \begin{array}{ccc} 0 & 1 & 0 \\
1 & 0 & 0 \\
0 & 0 & -1 \end{array} \right)
\end{eqnarray*}

\noindent Let $M^* = V' + W$ and define a bilinear form $h$ on $M^*$ so that $V' \perp W$, $h = f$ on $V'$ and $h = g$ on $W$. It is shown in \cite{J2} that $TKK(J)$ is isomorphic to the Lie algebra of endomorphisms of $M^*$ which are skew relative to $h$. 

\bigskip

\noindent We wish to generalize this construction for an infinite dimensional vector space. Suppose $V$ is any vector space over $F$ which is equipped with a nondegenerate symmetric form $f$ and $J$ is the Jordan algebra $F\cdot1 + V$. Then every finite collection of elements from $V$ can be embedded into a finite dimensional subspace $V'$ such that the restriction of $f$ to $V'$ is nondegenerate.

\bigskip

\noindent Let $\{V_{\alpha}\}_{\alpha \in I}$ be the collection of finite dimensional subspaces of $V$ with $f$ restricted to $V_{\alpha}$ nondegenerate. If $V_{\alpha} \subset V_{\beta}$, then we have $V_{\beta} = V_{\alpha} \oplus V_{\alpha}^{\perp}$. Then $I$ is a directed set with $\alpha \le \beta$ if $V_{\alpha} \subset V_{\beta}$. This is a local system of subspaces for $V$. Let $R_{\alpha} = End_F(V_{\alpha} + W)$. Then $h$ restricted to $V_{\alpha} + W$ is nondegenerate and symmetric. Then for $\alpha \le \beta$, we obtain an embedding $\phi_{\alpha, \beta}: R_{\alpha} \hookrightarrow R_{\beta}$ given by extending the action of each $x \in R_{\alpha}$ to an action on $V_{\beta}$ by having $x$ act trivially on $V_{\alpha}^{\perp}$. Notice that for $\alpha \le \beta \le \gamma$, $\phi_{\beta, \gamma} \circ \phi_{\alpha, \beta} = \phi_{\alpha \circ \gamma}$. Define $R$ to be the direct limit of $\{ R_{\alpha}, \phi_{\alpha, \beta}\}$. Then $R$ is a locally finite, simple associative algebra.

\bigskip

\noindent Let $K_{\alpha}$ denote the Lie algebra of elements of $R_{\alpha}$ which are skew with respect to the restriction of $h$ to $V_{\alpha} + W$. Each homomorphism $\phi_{\alpha, \beta}$ restricts to a Lie algebra homomorphism $\phi_{\alpha, \beta}: K_{\alpha} \hookrightarrow K_{\beta}$. Then $K = \varinjlim K_{\alpha}$ is the Lie algebra of elements of $R$ which are skew with respect to $h$.

\bigskip

\noindent Similarly, the set $\{ J_{\alpha} = V_{\alpha} + F\cdot1\}$ is a local system for $J$. For each $\beta \ge \alpha$, we have $V_{\alpha} \subset V_{\beta}$ and so we obtain an embedding $\rho_{\alpha, \beta}: TKK(J_{\alpha}) \hookrightarrow TKK(J_{\beta})$. Thus $\{TKK(J_{\alpha})\}$ is a local system for $TKK(J)$ and $TKK(J) = \varinjlim TKK(J_{\alpha})$. By \cite{J2}, for each $\alpha \in I$ there is an isomorphism $\psi_{\alpha}: K_{\alpha} \cong TKK(J_{\alpha})$ and one may check that $\psi_{\beta} \circ \phi_{\alpha, \beta} = \rho_{\alpha, \beta} \circ \psi_{\alpha}$. Thus, we obtain an isomorphism between $TKK(J)$ and $K$.


\bigskip

\noindent In case IV, the Lie algebra $L$ is finite dimensional over its centroid. By the results of section \ref{C}, the centroid of $L$ is an algebraic field extension. Since $F$ is algebraically closed, the centroid must be $F$ itself, in which case $L$ is a finite dimensional Lie algebra. Hence, since we have assumed that $L$ is infinite dimensional, case IV can be eliminated.

\section{\large \textbf{Proof of Theorem 2}}

\noindent In this section we provide the proof of Theorem 2. We begin with some preliminary results. The following lemma can be found in \cite{Pre}.

\begin{lemma}[Premet]\label{premet}
Suppose $p > 3$. Let $L$ be a simple, finite dimensional Lie algebra generated by elements satisfying $ad_x^3(L) = 0$. Then $L$ is of classical type. 
\end{lemma}

\noindent The next result will be used several times in the proof of Theorem 2.

\begin{lemma}\label{ideal}
Suppose $L$ is a Lie algebra over a field $F$ of characteristic $p$ (or zero) which is generated by a set $X$. Suppose for all $x \in X$, $ad(x)^d L = (0)$ for some $d > p$ and $exp(\xi ad(x))$ is a well-defined automorphism of $L$ for every $\xi \in F$. If $V$ is a subspace of $L$ which is invariant under Aut$(L)$, then $V$ is an ideal of $L$.
\end{lemma}

\begin{proof}
Fix $x \in X$ and $v \in V$. Since $L$ is generated by $X$, it suffices to show that $[x, v] \in V$. Choose  $\xi_1, \dots, \xi_d$ distinct nonzero elements from $F$. Since $ad(x)^{d} = 0$ and $exp(\xi_iad(x)) \in$ Aut$(L)$, we set $v_i = exp(\xi_iad(x))v \in V$ for $i = 1, \dots, d$. Consider the following system of $d$ linear equations with $d$ unknowns: $v$, $ad_xv$, $\frac{1}{2!} ad_x^2v$, . . . , $\frac{1}{(d-1)!}ad_x^{d-1}v $. 

\begin{eqnarray*}
v + \xi_1ad_xv + \xi_1^2 \frac{ad_x^2v}{2!} + \dots + \xi_1^{d-1}\frac{ad_x^{d-1}v}{(d-1)!} = v_1 \\
\vdots  \\
v + \xi_dad_xv+ \xi_d^2 \frac{ad_x^2v}{2!} + \dots + \xi_d^{d-1}\frac{ad_x^{d-1}v}{(d-1)!} = v_d \\
\end{eqnarray*}  

\noindent The matrix for this linear system is the Vandermonde matrix whose determinant is nonzero since $\xi_i \ne \xi_j$ for $i \ne j$. Therefore the system is invertible, and we conclude that each of the unknowns is an element of $V$, including $ad_x(v) = [x,v]$. 
\end{proof}

\noindent Recall that $L$ is \textit{uniformly ad-integrable} if the adjoint representation is uniformly integrable, that is, for each element $a \in L$, $ad(a)$ is a root of some polynomial $f_a(t) \in F[t]$. It follows that for each $a \in L$, $ad(a)$ has finitely many eigenvalues. The motivation for the study of such algebras follows from the fact that a simple, locally finite Lie algebra embeds into a locally finite associative algebra if and only if it is uniformly ad-integrable. 

\bigskip

\noindent Suppose $L$ is simple and uniformly ad-integrable. For $a \in L$, let $\mu_a(t)$ be the minimal polynomial of $ad(a)$. An element $a \in L$ is \textit{ad-nilpotent} if $ad(a)^NL = 0$ for some $N$, in which case the only root of $\mu_a(t)$ is 0. Define $d(L) = min\{ \mbox{deg $\mu_a(t)$} |  a\mbox{ is not ad-nilpotent} \}$. Since $L$ is simple, there is some element $a \in L$ which is not ad-nilpotent. Otherwise, by Engel's theorem $L$ would be locally nilpotent hence locally solvable, which contradicts lemma \ref{loc nilp}. Therefore, $d(L)$ is well-defined.

\bigskip

\noindent We now prove Theorem 2. 

\begin{theorem*}
\textup{Let $L$ be a simple, locally finite Lie algebra which is ad-integrable over an algebraically closed field $F$ of characteristic $p>3$. If $p > 2d(L) - 2$, then $L \cong [R^{(-)}, R^{(-)}] / (Z(R) \cap [R^{(-)}, R^{(-)}])$ where $R$ is a locally finite, simple associative algebra or $L \cong [K(R, *), K(R, *)]$ where $R$ is a locally finite, simple associative algebra with involution $*$.}
\end{theorem*}

\begin{proof}
Choose $a \in L$ so that deg $\mu_a(t) = d(L) = d$. Suppose $L = \sum_{\alpha \in \Phi} L_{\alpha}$ is the decomposition of $L$ into root spaces with respect to $a$. We see that $|\Phi| \le d(L)$. Set $\Delta = \{ 0 \ne \alpha \in \Phi | L_{\alpha} \ne (0)\}$, which is nontrivial since $a$ is not ad-nilpotent. By lemma \ref{simplegraded}, $L$ is generated by $\{ L_{\alpha} | \alpha \in \Delta \}.$ 

\bigskip

\noindent Choose $\alpha \in \Delta$, $\beta \in \Phi$, $x \in L_{\alpha}$, and $y \in L_{\beta}$. Since $p > 2d - 2$, $p > d$ as well. Thus $\beta, \beta + \alpha, \dots, \beta + d\alpha$ are all distinct roots. Since $|\Delta| \le d$, the sets $L_{\beta}, ad(x_{\alpha})L_{\beta}, ad(x_{\alpha})^2L_{\beta}, \dots, ad(x_{\alpha})^dL_{\beta}$ cannot all be nonzero, which implies that $ad(x_{\alpha})^dL_{\beta} = (0)$. Therefore, for every $x \in L_{\alpha}$, for $\alpha \in \Delta$, $ad(x)^d L = (0)$, and thus $exp(\xi ad(x))$ is a well-defined automorphism of $L$ for any $\xi \in F$. 

\bigskip


\noindent We have shown that for any $\alpha \in \Delta$ and $x \in L_{\alpha}$, $ad(x)^dL = 0$. Thus by lemma \ref{kostrikin}, there is some $0 \ne y \in L$ such that $ad(y)^3L = 0$, that is, L contains a Jordan element. Thus $V = \mbox{span}\{ x \in L | ad(x)^3L = 0\}$ is a nonzero subspace of $L$. Since $V$ is invariant under Aut$(L)$, we can apply lemma \ref{ideal} with $X = \{ L_{\alpha} | \alpha \in \Delta \}$ to show that $V$ is a nonzero ideal of $L$. Thus $V = L$ and $L$ is spanned, hence generated, by Jordan elements. 

\bigskip

\noindent Suppose $0 \ne x \in L$ is a Jordan element. Then if $J = L_x$ is \textit{not} locally nilpotent, then $J$ contains a nonzero idempotent and we proceed as in the proof of theorem 1. Suppose to the contrary that for every Jordan element $x \in L$, the Jordan algebra $L_{x}$ is locally nilpotent. 

\bigskip

\noindent Suppose $L_1 =$ Lie$(a_1, \dots, a_m)$ is a finite dimensional subalgebra generated by Jordan elements $a_1, \dots, a_m$ which is nonsolvable. Such an algebra must exist since $L$ is simple, so not locally solvable, and $L$ is generated by Jordan elements. Furthermore, suppose $L_1$ is such an algebra of minimal dimension. 

\bigskip

\noindent We will prove that $[L_1, L_1] = L_1$. Suppose not, that is, $[L_1, L_1] \ne L_1$. Let $V = \mbox{span}\{ x \in [L_1, L_1] | ad(x)^3L = 0\}$. Since $V$ is invariant under Aut$(L_1)$, $V$ is an ideal of $L_1$, hence $V$ is a subalgebra of $L_1$ generated by Jordan elements. By minimality of $L_1$, $V$ must be solvable. Thus we have shown that all Jordan elements from $[L_1, L_1]$ lie in the radical of $L_1$.

\bigskip

\noindent Let $R$ denote the solvable radical of $L_1$ and $\overline{L_1} = L_1 / R$. For each Jordan element $a_i$ and for each $b \in L_1$, $ad_{a_i}^2(b) \in [L_1, L_1]$ is a Jordan element by lemma \ref{identities}. Thus, $ad^2_{\overline{a_i}}\overline{L_1} = (0)$, so $\overline{L_1}$ is a finite dimensional Lie algebra generated by sandwich elements, that is, by elements $a \in \overline{L_1}$ such that $ad(a)^2\overline{L_1} = (0)$. By \cite{KZ}, $\overline{L_1}$ is nilpotent, thus $L_1$ is solvable, which is a contradiction.

\bigskip

\noindent We have proved that $L_1 = [L_1, L_1]$. Choose a maximal ideal $I \triangleleft L_1$. Then $L'_1 = L_1/I$ has no proper ideals. Since $L_1 = [L_1, L_1]$, $L'_1$ is not abelian. Therefore, $L'_1$ is a simple finite dimensional Lie algebra. Since $L'_1$ is generated by elements $a$ satisfying $ad^3_a(L'_1) = (0)$, by lemma \ref{premet} $L'_1$ must be classical hence nondegenerate. 

\bigskip

\noindent Recall that for each $a_i \in L_1$, we have assumed that the Jordan algebra $L_{a_i}$ is locally nilpotent. Thus $(L_1)_{a_i}$ and $(L'_1)_{a_i}$ are nilpotent. Fix $i$ such that $0 \ne \overline{a_i} \in L'_1$, which is possible since $I \ne L_1$. Since $(L'_1)_{a_i}$ is nilpotent, by proposition \ref{Kloc}, $0 \ne\overline{a_i}$ is in $K_{loc}(L'_1)$, hence $L'_1$ cannot be nondegenerate. This contradiction completes the proof.
\end{proof}

\bigskip

\noindent We remark that the reason we impose the condition $p> 2d(L)-2$ is to ensure that $exp(ad(x))$ is a well-defined automorphism of $L$ when $ad(x)^d = 0$. When the ground field has zero characteristic, this is always true. Thus, the proofs of theorems 1 and 2 give a new proof of the main result of \cite{BBZ}: Suppose $F$ is an algebrically closed field of characteristic zero and suppose $L$ is a simple, locally finite Lie algebra over $F$ which embeds into a locally finite associative algebra. Then $L$ is uniformly ad-integrable, and the proof of theorem 2 shows that $L$ contains an ad-nilpotent element. Then the proof of theorem 1 implies that $L \cong [R^{(-)}, R^{(-)}] / (Z(R) \cap [R^{(-)}, R^{(-)}])$ where $R$ is a locally finite, simple associative algebra or $L \cong [K(R, *), K(R, *)]$ where $R$ is a locally finite, simple associative algebra with involution $*$. In fact, both of these conclusions imply that $L \cong [K(R, *), K(R, *)]/ (Z(R) \cap [K(R, *), K(R, *)])$ for a locally finite, involution simple associative algebra $R$. Since $(Z(R) \cap [K(R, *), K(R, *)]) = (0)$ when the characteristic of the ground field is zero, we obtain the result of Bahturin, Baranov, and Zalesskii. 

\section{\large \textbf{Proof of Theorem 3}}

\noindent The purpose of this section is to address the following question: If $R$ is a simple, locally finite associative algebra with involution, is it true that $K(R, *) = [K(R, *), K(R, *)]$? Certainly if $R$ is locally simple, that is, $R = \bigcup_i R_i$ where each $R_i$ is simple, then $K(R, *) = [K(R, *), K(R, *)]$, as when $R = M_{\infty}(F)$. In this section we show this does not hold in general for a locally finite associative algebra $R$. Our construction works over an algebraically closed field of arbitrary characteristic.

\begin{theorem*}
There exists a simple, locally finite associative algebra $R$ with involution $*$ so that $K(R, *) \ne [K(R, *), K(R, *)]$.
\end{theorem*}

\begin{proof}
Let $F$ be an algebraically closed field, and let $p$ be an odd prime, $p = 2k+1$. Let $n_i = p^i$ for $i \ge 1$ and  $R_i = M_{n_i}(F) \oplus M_{n_i}(F)$ with involution $*$ given by $(A, B)^* = (B^T, A^T)$.  Let $\phi_i: R_i \hookrightarrow R_{i+1}$ be an embedding of signature $(k+1, k, 0)$, that is,

$$
\phi_i(M, N) =(diag(\underbrace{M, \dots, M}_{k+1}, \underbrace{N, \dots, N}_{k}),diag(\underbrace{N, \dots, N}_{k+1}, \underbrace{M, \dots, M}_{k}))
$$
 
\noindent It is easy to check that each $\phi_i$ is a map of involutive algebras. Define $R$ as the direct limit of the $R_i$ under each embedding, so $R = \varinjlim R_i$. Then $R$ is a locally finite, locally semisimple involutive algebra. 

\bigskip

\noindent We will show that $R$ is simple. We identify $R_i$ with its image in $R_{i+1}$. Suppose $I$ is a nonzero ideal of $R$ and $I_i = R_i \cap I$, so $I_{i} = I_{i+1} \cap R_i$ for all $i$. Since $I \ne (0)$, there is some $0 \ne (M,N) \in I_l$ for some $l$. Then $\phi_k(M, N) = (M', N')$ where $M'\ne 0$ and $N'\ne 0$. Since $M_{n_{i+1}}(F)$ is a simple ring, $id_{R_{l+1}}(M', N') = id_{R_{l+1}}(M', 0) + id_{R_{l+1}}(0, N') = M_{n_{l+1}}(F) \oplus (0) + (0) \oplus M_{n_{l+1}}(F) = R_{l+1}$. This shows that $R_{l+1} = I_{l+1}$ so $I_l = I_{l+1} \cap R_l = R_{l+1} \cap R_l = R_l$ and thus, $I_i = R_i$ for all $i \le l$. The above argument also demonstrates that $R_i = I_i$ for all $i \ge l$ as well, thus $I = R$. Therefore, $R$ must be simple.

\bigskip

\noindent We now show that $K(R, *) \ne [K(R, *), K(R, *)]$. Note that $K(R_i, *) = \{ (M, -M^T)\} \cong M_{n_i}(F)^{(-)}$. Choose $(M, -M^T) \in K(R_i, *)$ such that the trace of $M$ is nonzero. Then $\phi_i(M, -M^T) = (M_1, -M_1^T)$ and Tr$(M_1) = (k+1 - k)$Tr$(M) = $Tr$(M) \ne 0$. Therefore in any subalgebra of $R$, $(M, -M^T)$ is always mapped to an element $(M', -M'^T)$ such that $M'$ has nonzero trace.

\bigskip

\noindent Suppose $(M, -M^T) \in [K(R, *), K(R, *)]$, so we have that 

$$(M, -M^T) = \sum_j [(a_j, -a_j^T), (b_j, -b_j^T)] = (\sum_j [a_j, b_j], \sum_j [-a_j^T, -b_j^T])$$

\noindent in some subalgebra $R_l$, where $(M, -M^T)$ is identified with its image in $R_l$. But $M$ has nonzero trace, whereas $\sum_j [a_j, b_j]$ must be of trace zero. This is a contradiction, therefore $(M, -M^T)$ cannot be written as a sum of commutators. Therefore, $K(R, *) \ne [K(R, *), K(R, *)]$.
\end{proof}

\bigskip

\noindent Lastly, we note that the algebra $R$ we constructed in the previous proof has symmetry type ``two-sided weakly non-symmetric" according to \cite{Ba}. Therefore, by theorem 5.2 of $\cite{Ba}$, $R$ cannot be isomorphic to a direct limit of involution simple algebras of orthogonal or symplectic type, so $R$ is not locally simple. This is consistent with our result that $K(R, *) \ne [K(R, *), K(R, *)]$.

\small{

\bigskip

\textsc{University of California, San Diego, Department of Mathematics, 9500 Gilman Drive
La Jolla, CA 92093-0112}

\medskip

\noindent \textit{Email address:} \texttt{jhennig@math.ucsd.edu}}

\end{document}